\documentclass{svmult}
\usepackage{amsmath,amssymb,amsfonts,color}
\def\Re{\mathbb{R}}

\def\Sec#1{Sec.~\ref{#1}}

\def\notes#1{\marginpar{\tiny #1}\typeout{Notes!
Notes!
Notes!
}}
\renewcommand{\notes}[1]{\typeout{notes!}}
\def\FRAC#1#2#3{\genfrac{}{}{}{#1}{#2}{#3}}
\def\half{{\mathchoice{\FRAC{1}{1}{2}}%
{\FRAC{2}{1}{2}}%
{\FRAC{3}{1}{2}}%
{\FRAC{4}{1}{2}}}}

\def\Re{\field{R}}

\def\Sec#1{Sec.~\ref{#1}}

\def\clB{{\cal B}}
\def\clE{{\cal E}}

\def\clP{{\cal P}}
\def\clZ{{\cal Z}}

\def\Sec#1{Sec~\ref{#1}}

\def\E{{\sf E}}

\def\Sec#1{Sec.~\ref{#1}}

\def\clZ{{\cal Z}}

\def\beq{\begin{eqnarray}} 
\def\bc{\begin{center}} 
\def\be{\begin{enumerate}}
\def\bi{\begin{itemize}} 
\def\bs{\begin{small}}
\def\bS{\begin{slide}}
\def\ec{\end{center}} 
\def\ee{\end{enumerate}}
\def\ei{\end{itemize}}
\def\es{\end{small}}
\def\eS{\end{slide}}
\def\eeq{\end{eqnarray}}

\newcommand{\newP}[1]{\medskip\noindent{\bf #1:}}

\newcommand{\ud}{\,\mathrm{d}}

\def\Re{\mathbb{R}}
\def\E{{\sf E}}

\def\Sec#1{Sec.~\ref{#1}}
\def\Thm#1{Thm.~\ref{#1}}
\def\Prop#1{Prop.~\ref{#1}}

\def\clB{{\cal B}}
\def\clD{{\cal D}}
\def\clE{{\cal E}}

\def\clP{{\cal P}}
\def\clZ{{\cal Z}}

\renewcommand{\Re}{\mathbb{R}}

\def\FRAC#1#2#3{\genfrac{}{}{}{#1}{#2}{#3}}

\newcommand{\var}{\text{var}}

\def\clA{{\cal A}}
\def\clB{{\cal B}}

\def\clD{{\cal D}}
\def\clE{{\cal E}}
\def\clF{{\cal F}}

\def\clM{{\cal M}}

\def\clP{{\cal P}}

\def\clV{{\cal V}}

\def\clZ{{\cal Z}}
\def\E{{\sf E}}
\def\bS{\mathbb{S}}

\def\ones{{\sf 1}}
\def\sP{{\sf P}}

\def\tp{{\hbox{\rm\tiny T}}}

\def\bmu{{\bar\mu}}

\def\chisq{{\chi^2}}
\def\kl{{\sf D}}
\def\tv{{\text{\tiny TV}}}

\def\fish{{\sf I}}
\def\gS{\bS}

\begin{document}

\title{Divergence metrics in the study of Markov and hidden Markov processes}

\titlerunning{Divergence metrics}
% Use \titlerunning{Short Title} for an abbreviated version of
% your contribution title if the original one is too long

\author{Jin Won Kim \inst{1},
Amirhossein Taghvaei\inst{2},\and
Prashant G. Mehta\inst{3}}
% Use \authorrunning{Short Title} for an abbreviated version of
% your contribution title if the original one is too long

\institute{Department of Mechanical \& System Design Engineering, Hongik University
\texttt{jin.won.kim@hongik.ac.kr}
\and Department of Aeronautocs and Astronautics, University of Washington Seattle  \texttt{amirtag@uw.edu}
\and Department of Mechanical Science and Engineering, Coordinated Science Laboratory, University of Illinois at Urbana-Champaign  \texttt{mehtapg@illinois.edu}
}
%
% Use the package "url.sty" to avoid
% problems with special characters
% used in your e-mail or web address
%

\maketitle

\vspace*{0.5in}

\noindent 
\emph{Dedicated to the memory of Roger Brockett}

\vspace*{0.5in}

%\tableofcontents

\begin{abstract}
This paper is divided into two parts. The first part reviews the formulae for f-divergences in the study of continuous-time Markov processes and explores their applications in areas such as stochastic stability, the second law of thermodynamics, and its non-equilibrium extensions. This sets the foundation for the second part, which focuses on f-divergence in the study of hidden Markov processes. In this context, we present analyses of filter stability and stochastic thermodynamics, with the latter being used to illustrate the concept of a Maxwell demon in an over-damped Langevin model with white noise observations. The paper’s expository style and unified formalism for both Markov and hidden Markov processes aim to serve as a valuable resource for researchers working across related fields.
\end{abstract}

\section{Introduction}

This expository paper discusses the use of f-divergence metrics in the
study of continuous-time Markov and hidden Markov processes.  The most
classical type of the f-divergence metric is the Kullback-Leibler (KL)
divergence
(or relative entropy) which is
codified in the second law of thermodynamics.  Relative entropy is
useful also in the study of asymptotic stability of a Markov process, based upon certain addition assumptions, namely, existence of an
invariant measure that satisfies the Log Sobolev
Inequality (LSI).

The first part of this paper is concerned with Markov processes.  The
goal is to introduce and present the formulae for the important types
of f-divergence metrics, namely, 
KL divergence, $\chisq$-divergence, and total variation.  Applications of
these formulae are discussed related to (i) stochastic stability of
the Markov process, (ii)
second
law of thermodynamics for an over-damped Langevin dynamics, and (iii) modifications of the second law in
the non-equilibrium settings. 
While all of these formulae and
ensuing analysis is classical, its unified presentation is meant to be useful
for researchers working in disparate communities, specifically theorists
interested in thermodynamics formalisms.  Of topical interest is the
non-equilibrium thermodynamics; cf.,~\cite{lebon2008understanding,seifert2008stochastic,sekimoto2010stochastic,seifert2012stochastic,parrondo2015thermodynamics,brockett2017thermodynamics}.
The analysis of the KL and $\chisq$-divergences together is useful to see the motivation for the
Poincar\'e Inequality (PI) and the LSI familiar in the study of Markov
processes.  Such inequalities are an important theme in optimal transportation
theory~\cite{villani2008optimal,bakry2013analysis}.

Formulae and analysis presented in the first part sets the stage for the second part where the focus is on the hidden Markov models (HMM). A key distinction is that, for an HMM, the f-divergence is not always monotonic. In other words, it is not necessary that f-divergence is a non-increasing function over time. While this holds for various types of f-divergences, the KL-divergence is known to remain monotonic even for an HMM~\cite{clark1999relative}.

The second part of this paper presents a self-contained discussion on the various definitions of filter stability. It reviews prior work on analyzing filter stability for each of the three types of f-divergences. Specifically, the analysis of total variation is explained through the forward map, which is fundamental to the {\em intrinsic approach} for studying filter stability. This approach is contrasted with an alternate approach based on the backward map, introduced by two of the authors (JK and PM).

Building on the formulas for f-divergence, the second part of this paper also describes extensions of classical stochastic thermodynamics to the context of HMMs. These extensions are based on recent results from one of the authors (AT). It is demonstrated that information from observations can be used to extract additional work, a significant theme in recent literature. This relates to the apparent violation of the second law in the presence of information-carrying observations, a topic of both historical and current interest~\cite{sagawa2008discreteqmeas,Mauro2020contqmeas,mitter2005information,Horowitz_2014,Abreu_2011workfeedback,Bauer_2012infomachine,parrondo2021extracting,taghvaei2021relation}.

The outline of the remainder of this paper is as follows:
Sec.~\ref{sec:MP} contains the first part of this paper where formulae
for Markov processes are described together with 
applications to stochastic stability and the 
second law of thermodynamics. 
Sec.~\ref{sec:HMM} introduces the problem of filter stability for a
general class of HMMs and gives an overview of analysis
approaches for the three types of f-divergences. 
Sec.~\ref{sec:white-noise} introduces the specific model of white
noise observations where explicit formulae can be obtained.  These
formulae are
described together with applications to filter stability and
stochastic thermodynamics.
%Sec.~\ref{sec:bmap_opt}

\subsection{Notation}

For a locally compact Polish space $\gS$, the following notation is adopted:
\begin{itemize}
	\item $\clB(\gS)$ is the Borel $\sigma$-algebra on $\gS$.
	\item $\clM(\gS)$ is the space of regular, bounded and finitely additive 
	signed measures (rba measures) on $\clB(\gS)$.
	\item $\clP(\gS)$ is the subset of $\clM(\gS)$ comprising of probability 
	measures.
	\item $C_b(\gS)$ is the space of continuous and bounded real-valued functions on $\gS$.
	\item For measure space $(\gS;\clB(\gS);\lambda)$, $L^2(\lambda)= L^2(\gS;\clB(\gS);\lambda)$ is the Hilbert space of real-valued functions on $\gS$ equipped with the inner product
	\[
	\langle f, g\rangle_{L^2(\lambda)} = \int_\gS f(x)g(x)\ud \lambda(x)
	\]
\end{itemize}

For functions $f:\gS\to \Re$ and $g:\gS\to \Re$, the notation $fg$ is used to denote element-wise 
product of $f$ and $g$, namely,
\[
(fg)(x) := f(x)g(x),\quad x\in \gS
\]
In particular, $f^2 = ff$. The constant function is denoted by $\ones$ ($\ones(x) = 1$ for all $x\in \gS$).

For $\mu\in \clM(\gS)$ and $f\in C_b(\gS)$,
\[
\mu(f) := \int_\gS f(x) \ud \mu(x)
\]
and for $\mu,\nu\in \clM(\gS)$ such that $\mu$ is absolutely continuous with 
respect to $\nu$ (denoted $\mu\ll\nu$), the Radon-Nikodym (RN)
derivative (or likelihood ratio) is 
denoted by $\dfrac{\ud \mu}{\ud \nu}=:\gamma$ ($\gamma$ is a real-valued function on $\bS$). The three types of
f-divergences, that are of most interest to this paper, are as follows:
	\begin{align*}
		\text{(KL divergence)}\qquad& \kl(\mu\mid \nu) := \nu(\gamma\log\gamma)=\int_\bS \gamma(x) \log(\gamma(x))\ud \nu(x)\\
		\text{($\chi^2$ divergence)}\qquad& \chisq(\mu\mid \nu) := \nu((\gamma-1)^2)=\int_\bS (\gamma(x) - 1)^2\ud \nu(x)\\
		\text{(Total variation)}\qquad& \|\mu-\nu\|_\tv := \half \nu(|\gamma-1|)=\int_\bS \half |\gamma(x)-1|\ud \nu(x)
	\end{align*}
        
Throughout this paper, we consider continuous-time processes on a
finite time horizon $[0,T]$ with $T<\infty$. Fix the probability space
$(\Omega, \clF_T, \sP)$ along with the filtration $\{\clF_t:0\le t \le
T\}$ with respect to which all the stochastic processes are adapted.

\section{Markov processes}\label{sec:MP}

\subsection{Model and math preliminaries}\label{ssec:model-X}

The notation and the model for the Markov process are as follows:
\begin{itemize}
	\item The \emph{state process} $X = \{X_t:0\le t \le T\}$ is a $\bS$-valued Feller-Markov process. Its initial measure (referred to as the prior) is denoted by $\mu \in \clP(\bS)$ and $X_0\sim \mu$. The infinitesimal generator of the Markov process is denoted by $\clA$. In terms of $\clA$, the \emph{carr\'e du champ} operator $\Gamma$ is a bilinear operator defined by (see~\cite[Defn.~3.1.1]{bakry2013analysis}),
	\[
	\Gamma (f,g)(x) := \big(\clA fg)(x) - f(x)\big(\clA g)(x) - g(x)\big(\clA f)(x),\quad x\in \bS
	\]
	for every $(f,g)\in\clD\times \clD$. Here, $\clD$ is a vector space of (test) functions that are dense in a suitable $L^2$ space, stable under products (i.e., $\clD$ is an algebra), and $\Gamma:\clD\times \clD\to \clD$ (i.e., $\Gamma$ maps two functions in $\clD$ into a function in $\clD$), such that $\Gamma(f,f)\geq 0$ for every $f\in\clD$. For the case where an invariant measure $\bmu\in\clP(\bS)$ is available then the natural $L^2$ space is with respect to the invariant measure: $L^2(\bmu) = \{f:\bS\to\Re: \bmu(f^2)<\infty\}$. 
	A sample path $t \mapsto X_t(\omega)$ is a $\bS$-valued c\`adl\`ag function (that is right continuous with left limits).  The space of such functions is denoted by $D\big([0,T];\bS\big)$. In particular, $X(\omega)\in D\big([0,T];\bS\big)$ for each $\omega \in \Omega$.
\end{itemize}
The model is succinctly denoted as $X=\text{Markov}(\clA,\mu)$. 
For the case where $\bS$ is not finite, additional assumptions are 
typically necessary to ensure that the model is well-posed.  In lieu
of stating these conditions for a general class of Markov processes,
we restrict our study to the guiding examples described in the following.

\subsection{Guiding examples for the Markov processes}\label{ssec:guiding}

The most important examples are as follows: (i) the state space $\bS$ is finite of cardinality $d$; and (ii) the state space $\bS$ is the Euclidean space $\Re^d$. With $\bS=\Re^d$, the linear Gaussian Markov process is of historical interest. In the following, notation and additional assumptions are introduced for these two type of examples. An important objective is to describe the explicit form of the carr\'e du champ operator for each of these examples.

\newP{Finite state space} The state-space $\bS$ is finite, namely, $\bS= \{1,2,\ldots,d\}$. 
In this case, the function space and measure space are both isomorphic to $\Re^d$: a real-valued function $f$ (resp., a finite measure $\mu$) is identified with a vector in $\Re^d$, where the $x^{\text{th}}$ element of the vector equals $f(x)$ (resp., $\mu(x)$) for $x\in\bS$. In this manner, an $\Re^m$-valued function $h$ is identified with a matrix $H\in\Re^{d\times m}$. 
The generator $\clA$ of the Markov process is identified with a transition rate matrix $A\in\Re^{d\times d}$ (the non-diagonal elements of $A$ are non-negative and the row sum is zero). $\clA$ acts on
a function $f$ through right-multiplication by the matrix $A$:
\[
\clA: f\mapsto A f
\]
The carr\'e du champ operator $\Gamma:\Re^d\to \Re^d$ is as follows:
\begin{equation}\label{eq:Gamma-finite}
	(\Gamma f)(x) = \sum_{y \in \bS} A(x,y) (f(x) - f(y))^2,\quad x\in\bS
\end{equation}

\begin{remark}
	The notation for the finite state space case is readily extended to the countable state space $\bS = \{1,2,\ldots\}$. In this case, $A = \{A(x,y): x,y\in \bS\}$ and the carr\'e du champ is again given by the equation~\eqref{eq:Gamma-finite}.
	Typically, additional conditions are needed to ensure that
        the Markov process $X$ is well-posed over the time horizon
        $[0,T]$. The simplest such condition is that $A$ has bounded rates, i.e., $\sup_{i\in \bS} \sum_{j\neq i} A(i,j) < \infty$. 
\end{remark}

\newP{Over-damped Langevin dynamics} The state space $\bS=\Re^d$. In the study of stochastic thermodynamics, the over-damped Langevin dynamics model is of interest.  The model is given by the stochastic differential equation (SDE),
\begin{equation}\label{eq:dyn_sde}
	\ud X_t = -\nabla U (X_t) \ud t + \sqrt{2}\ud B_t, \quad X_0\sim \mu
\end{equation}
where the initial measure $\mu$ is assumed to have a well-defined probability density function (pdf) on $\bS$ (the pdf is denoted by $\frac{\ud \mu}{\ud x}$), $-\nabla U(x)$ is the force exerted, at the point $x\in\bS$, due to the potential $U\in C^2(\Re^d; \Re)$, and $B=\{B_t:0\le t \le T\}$ is a standard B.M.~assumed to be independent of $X_0$.  It is assumed that the potential $U$ satisfies appropriate growth conditions at $\infty$ such that a strong solution exists for all $t\in[0,T]$. In writing the SDE model~\eqref{eq:dyn_sde}, the temperature and other thermodynamic coefficients have been normalized to one (for additional motivation and background on the model see~\cite{sekimoto2010stochastic,peliti2021stochastic}).

The infinitesimal generator $\clA$ acts on a function $f$ in its domain $C^2(\Re^d;\Re)$ according to~\cite[Thm. 7.3.3]{oksendal2003stochastic}
\[
(\clA f)(x):=- \nabla U(x)^\tp \nabla f(x) + \Delta f(x),\quad x\in\Re^d
\]
where $\nabla f$ is the gradient vector and $\Delta f$ is the Laplacian of the function $f$. 
For $f\in \clD:=C^1(\Re^d;\Re)$, the carr\'e du champ operator is given by 
\begin{equation*}\label{eq:Gamma-Euclidean}
	(\Gamma f) (x) = 2\big| \nabla f(x) \big|^2,\quad x\in\Re^d
\end{equation*}

\newP{Linear-Gaussian model} A linear-Gaussian model is obtained for a quadratic choice of the potential function, $U(x) = \frac{1}{2}x^\tp K x$, where $K$ is a positive-definite matrix (denoted $K\succ 0$), and a Gaussian prior.  Then the SDE~\eqref{eq:dyn_sde} becomes
\[
	\ud X_t = -K X_t \ud t + \sqrt{2} \ud B_t,\quad X_0\sim N(m_0,\Sigma_0) 
\]
where $N(m_0,\Sigma_0)$ is the notation for a Gaussian measure with mean $m_0\in \Re^d$ 
and variance $\Sigma_0 \succ 0$.

With a slight abuse of notation, a linear function is expressed as \[f(x) = f^\tp 
x,\quad x\in\Re^d\] 
where on the right-hand side $f\in \Re^d$. Then $\clA f$ is a linear function 
given by,
\[
\big({\clA} f\big)(x) = -f^\tp K x,\quad x\in\Re^d
\]
and $\Gamma f$ is a constant function given by,
\[
\big(\Gamma f\big)(x) = 2|f|^2,\quad x\in\Re^d %= f^\tp Q 
%f
\]

\begin{remark}\label{rem:non-eqm}
  In the study of non-equilibrium thermodynamics, the basic model~\eqref{eq:dyn_sde} is extended to consider time-varying potentials.  A time-varying potential is denoted by $U_t(x)$.  In the linear Gaussian setting, $U_t(x) = \frac{1}{2}x^\tp K_t x$ where $K_t\succ 0$ for each time $t$. The formula for the generator $\clA$ is obtained through replacing the potential $U$ (resp.,~the matrix~$K$) with $U_t$
(resp.,~$K_t$).  
\end{remark}

\subsection{Formulae for $f$-divergence}

In the study of Markov processes, the basic object of interest is the expectation
\[
\mu_t(f) := \E (f(X_t)),\quad f\in C_b(\bS)
\]
The expected value depends upon the choice of the initial measure.  With two different initializations, $\mu$ and $\nu$, the formula for expectation is used to define two $\clP(\bS)$-valued deterministic 
processes $\{\mu_t:t\geq 0\}$ and $\{\nu_t:t\geq 0\}$.  These are the 
solution of the forward Kolmogorov equations as follows:
\begin{align*}
	\frac{\ud }{\ud t} \,\mu_t(f) &= \mu_t\big(\clA f\big),\quad \mu_0=\mu\\
        \frac{\ud }{\ud t} \,\nu_t(f) &= \nu_t\big(\clA f\big),\quad \nu_0=\nu
\end{align*}
Supposing $\mu\ll\nu$ then it is true that $\mu_t\ll \nu_t$ (see Lemma~\ref{lm:change-of-Pmu-Pnu}).  Denote the RN derivative
\[\gamma_t(x):= \frac{\ud 
  \mu_t}{\ud \nu_t}(x),\quad x\in\bS,\quad t\geq 0
\]
The formulae for the
              two types of f-divergences appear in the following
              theorem.  (The total variation is not $C^1$ in its
              argument and therefore not amenable to taking the
              derivative). 
\begin{theorem}\label{thm:div_MP}
%Markov process.
\begin{align*}
\frac{\ud }{\ud t} \,\kl(\mu_t\mid\nu_t) &= -\frac{1}{2}\nu_t\Big(\frac{\Gamma \gamma_t}{\gamma_t}\Big),\quad t\geq 0 \\ 
\frac{\ud }{\ud t} \,\chisq(\mu_t\mid\nu_t) &= -\nu_t\big(\Gamma 
	\gamma_t\big),\quad t\geq 0
\end{align*}
where $\Gamma$ is the carr\'e du champ operator. 
\end{theorem}
\begin{proof}
See Appendix~\ref{apdx:pf-thm-div-MP}.
\end{proof}

An important point to note is that the right-hand side for each of the two derivatives is non-positive (because of the positive-definite property of the carr\'e du champ operator $\Gamma$). As a result, the two divergences both serve as candidate Lyapunov functions.  

\subsection{Application to stochastic stability}

One would like to use the formulae in \Thm{thm:div_MP} to obtain
useful conclusions on asymptotic properties of the measure-valued
process $\{\mu_t:t\geq 0\}$ .  For this purpose, let us first suppose that 
there exists an invariant measure $\bmu\in\clP(\bS)$. As discussed at the onset, the domain of $\Gamma$ is denoted by $\clD$ and assumed to be dense in $L^2(\bmu)$.

\medskip

\noindent \textbf{Stochastic stability using $\chi^2$-divergence.} 
For asymptotic analysis based on $\chisq$-divergence, the \emph{Poincar\'e
  inequality} (PI) is defined as follows: For $f\in \clD$, 
	\begin{align*}
		\text{(energy):}\qquad\clE^\bmu(f) &:= \bmu\big(\Gamma f)\\
		\text{(variance):}\qquad\clV^\bmu(f) &:= \bmu (f^2) -\bmu(f)^2 
	\end{align*}
The Poincar\'e constant is defined as
\[
c = \inf\{\clE^\bmu(f) \; :\; f\in \clD \; \;\& \; \; \clV^\bmu(f) = 1\}
\]
When the Poincar\'e constant $c$ is strictly positive the resulting inequality is referred to as the
Poincar\'e inequality (PI):
	\begin{equation}\label{eq:standard-PI}
		\text{(PI)}\qquad\qquad	\clE^\bmu(f) \geq c \;
		\clV^\bmu(f),\quad \forall\, f \in \clD
	\end{equation}
PI is a fundamental functional inequality in applied
mathematics. 
Its significance to the problem of stochastic stability is as
follows: Set $\nu = \bmu$.  Then because $\bmu$ is an invariant measure, $\nu_t=\bmu$, and therefore, $\gamma_t = \frac{\ud\mu_t}{\ud\bmu}$ for $t\geq 0$.
        Noting $\chisq(\mu_t\mid\bar{\mu})
        =\bmu(\gamma_t^2)-1=\clV^\bmu(\gamma_t)$, the differential equation for $\chisq$-divergence becomes 
		\[
	\frac{\ud}{\ud t} \clV^\bmu(\gamma_t) = -\clE^\bmu(\gamma_t)
        \stackrel{\text{(PI)}}{\leq} -c \; \clV^\bmu(\gamma_t) 
	\]
	and using Gronwall,
	\[
	\chisq(\mu_t\mid\bar{\mu}) \le e^{-ct} \; \chisq(\mu \mid\bar{\mu}),\quad \forall \;t\geq 0
	\]   
Therefore, for all such initializations $\mu\in\clP(\bS)$ with $\chisq(\mu \mid\bar{\mu})< \infty$, asymptotic
stability in the sense of $\chi^2$-divergence is shown. 

\medskip

\noindent \textbf{Stochastic stability using KL-divergence.} 
The role that PI plays for the analysis of $\chisq$-divergence, the logarithmic Sobolev
Inequality (LSI) plays the same role for the analysis of
KL-divergence.  Specifically, define
	\begin{align*}
		\text{(Fisher information):}\qquad \fish(\mu \mid\nu )
          := \frac{1}{2}\nu \left(\frac{\Gamma \gamma}{\gamma} \right)
	\end{align*}
(where note $\gamma = \frac{\ud\mu}{\ud\nu}$). 
The relationship between the KL divergence and the Fisher information is asserted by the \emph{logarithmic Sobolev inequality} (LSI):
\begin{equation}\label{eq:log-Sobolev}
	\text{(LSI)}\qquad\qquad	\fish(\mu\mid\bmu)\geq c'\,\kl(\mu\mid\bmu) ,\quad \forall \mu\ll\bmu
\end{equation}
Assuming LSI holds, the proof for stochastic stability is immediate: 
As before, set $\nu
= \bmu$.  Then 
\begin{align*}
\frac{\ud }{\ud t} \,\kl(\mu_t\mid\bmu) \leq -\fish(\mu_t   \mid\bmu )  \stackrel{\text{(LSI)}}{\leq} -c' \; \kl(\mu_t\mid\bmu) 
\end{align*}  and using Gronwall
\[
\kl(\mu_t\mid\bmu) \le e^{-c't} \kl(\mu\mid\bmu)
\]
Therefore, provided $\kl(\mu\mid\bmu) <\infty$, asymptotic stability in the sense of KL divergence is obtained.

The relationship between the PI and the LSI is given in the following:
\begin{proposition}[Prop.~5.1.3 of~\cite{bakry2013analysis}]
	If $(\bmu,\Gamma)$ satisfies LSI with constant $c'$, then the pair also satisfies PI with the constant $c=2c'$.
\end{proposition}
\begin{proof}
	See~\cite[Prop.~5.1.3]{bakry2013analysis}.
\end{proof}

\subsection{Application to the second law of thermodynamics} \label{sec:2nd-law}

A colloquial understanding of the second law of thermodynamics is that ``entropy never decreases over time''.  This is made precise for the over-damped
Langevin dynamics model~\eqref{eq:dyn_sde}.  Suppose $C:=\int
e^{-U(x)} \ud x < \infty$. Then the model admits an absolutely continuous (with respect to Lebesgue) 
invariant measure whose density is given by
\[
\frac{\ud \bar \mu}{\ud x}(x) = \frac{1}{C}e^{-U(x)}, \quad x\in\Re^d
\]
(To see this, use integration by parts to show $\bmu({\cal A} f) = \frac{1}{C}\int_{\Re^d} {\cal A}f(x) e^{-U(x)} \ud x = 0$ for all $f\in\clD$). This is referred to as the Boltzmann relationship where the normalization constant $C$ is called the partition function (in the more standard notation denoted by $Z$).   

For any absolutely continuous measure $\mu\in {\cal M}(\bS)$ with (Lebesgue) density $\frac{\ud \mu}{\ud x}$, we have the
following
definitions~\cite{sekimoto2010stochastic,peliti2021stochastic}: 
\begin{align*}
	\text{(thermal energy):}\qquad\quad\;\, \mu(U) & := \int U(x)
                                                         \frac{\ud
                                                         \mu}{\ud x}
                                                         (x) \ud x \\
	\text{(entropy):}\qquad\quad\;\; \mathcal S(\mu) &:= - \int
	\frac{\ud \mu}{\ud x}(x) \log (\frac{\ud \mu}{\ud x}(x) )
	\ud x \\
	\text{(free energy):}\qquad\, \mathcal  F(\mu,U) &:= \mu(U) -
	\mathcal S(\mu) 
\end{align*}
The reader is warned that a notion for ``energy'' has already been defined for the carr\'e du champ.  The thermodynamic notion of energy, referred to as ``thermal energy'' above, is not the same.  Luckily, in the thermodynamics context, the more important notion is the free energy which has a ready counterpart with the preceding analysis, specifically,
\begin{align*}
	\mathcal F(\mu_t,U) =\kl(\mu_t\mid\bmu)  -  \log(C)
\end{align*}
Therefore, using the formula in \Thm{thm:div_MP},
\begin{align}\label{eq:dF-identity}
	\frac{\ud}{\ud t} 	\mathcal  F(\mu_t,U)   = 	\frac{\ud}{\ud t} \kl(\mu_t\mid\bmu) = -\fish(\mu_t\mid\bmu)\leq  0
\end{align}
This shows that the free energy is non-increasing. 
Upon
integrating the differential equation from the initial time $t=0$ till a terminal time $T$, 
\begin{align*}
	-\Delta \mathcal F  : = -(\mathcal F(\mu_{T},U) - \mathcal F(\mu_{0},U))  =  \int_0^{T}  
	\fish(\mu_t\mid\bmu) \ud t\geq 0
\end{align*} 
This is the second law of thermodynamics which states that the change in the total entropy $\Delta \mathcal S_\text{tot}:= - \Delta \mathcal F \geq 0$.

\subsection{Extensions to non-equilibrium thermodynamics}
In the study of non-equilibrium thermodynamics, one considers a time-varying counterpart of the over-damped Langevin dynamics~\eqref{eq:dyn_sde} (as discussed in Remark~\ref{rem:non-eqm}).  The time-varying model is given by,
\begin{equation}\label{eq:dyn_sde_TV} 
\ud X_t = -\nabla U_t (X_t) \ud t + \sqrt{2}\ud B_t, \quad X_0\sim \mu
\end{equation}
whose generator is defined by
\[
(\clA_t f)(x):=- \nabla U_t(x)^\tp \nabla f(x) + \Delta f(x),\quad x\in\Re^d
\]
for any function $f\in C^2(\Re^d;\Re)$. The time-varying potential may be regarded as the effect of an input that drives the model system. 

In the study of the time-varying model, two $\clP(\bS)$-valued deterministic processes $\{\mu_t:t\geq 0\}$ and $\{\nu_t:t\geq 0\}$ are considered. The former is defined as the  solution of the forward Kolmogorov equation (same as before)
\begin{align*}
	\frac{\ud }{\ud t} \,\mu_t(f) &= \mu_t\big(\clA_t f\big),\quad \mu_0=\mu
\end{align*}
and the latter is defined by the Boltzmann relationship 
\begin{equation*}
	\frac{\ud \nu_t}{\ud x}(x) =\frac{1}{C_t}e^{-U_t(x)}, \quad t\geq 0,
\end{equation*}
where $C_t:=\int e^{-U_t(x)}\ud x$ (it is assumed to be finite for all
$t\geq 0$). The Boltzmann density satisfies $\nu_t (\clA_t f) = 0 $ for all $f\in \clD$, for all $t\geq 0$.  Provided the potential $U_t$ is slowly time-varying, the measure $\nu_t$ has the meaning of a ``quasi-static'' invariant measure. In this quasi-static limit, the measure $\mu_t$ converges towards $\nu_t$ and such transitions are referred to as non-equilibrium transitions.   

With a time-varying potential, the differential equation~\eqref{eq:dF-identity} for the free energy is modified to
\begin{align*}
	\frac{\ud}{\ud t} 	\mathcal  F(\mu_t,U_t)=\frac{\ud}{\ud
  t} \kl(\mu_t\mid\nu_t)  - \frac{\ud}{\ud t} \log (C_t) = -\fish(\mu_t\mid\nu_t) + \mu_t (\frac{\partial}{\partial t} U_t)
\end{align*}          
where $\mu_t (\frac{\partial}{\partial t} U_t) = \int_{\bS} \frac{\partial}{\partial t} U_t(x) \ud \mu_t(x)$
is called the rate of work input. Its integral over a time-horizon is
referred to as the work:
\[
\text{(work):}\qquad 	\mathcal  W_T := \int_0^T \mu_t (\frac{\partial}{\partial t}  U_t) \ud t
\]
In the presence of work, the modified form of the second law of thermodynamics, obtained from integrating the differential equation for the free energy (as before), is as follows:
\begin{equation}\label{eq:2nd-law-ineq}
	\mathcal W_{T} - \Delta \mathcal F= \int_0^{T}\fish(\mu_t\mid\nu_t)  \ud t \geq 0
\end{equation}
The difference $\mathcal W_{T} - \Delta \mathcal F$ is the total entropy increase, and the positive term $\fish(\mu_t\mid\nu_t)$ is referred to as the rate of entropy production or dissipation.   A fundamental problem is to
obtain bounds on the maximal amount of work that can be extracted from
non-equilibrium
transitions~\cite{schmiedl2007efficiency,dechant2017underdamped,chen2019stochastic,fu2020maximal}. 

For the over-damped Langevin model, the dissipation rate admits an explicit formula
\begin{align*}
	\int_0^{T}\fish(\mu_t\mid\nu_t)
	\ud t = \int_0^{T} \mu_t(|v_t|^2)\ud t 
\end{align*} 
where 
\[
v_t(x):= - \nabla U_t(x)  - \nabla \log (\frac{\ud \mu_t}{\ud
	x}(x)),\quad x\in\Re^d 
\]
This representation is used to provide an interpretation for the dissipation  in terms of an optimal transport Riemannian metric in the space of probability densities~\cite{chen2019stochastic,fu2020maximal}.

\section{Hidden Markov model}\label{sec:HMM}

\subsection{Model and math preliminaries}

The second part of this paper is concerned with the study of hidden Markov model (HMM). A HMM is a Markov model in which the state process is
hidden from direct observation but a state-dependent observation
process is available. The state and observation processes are defined as
follows:
\begin{itemize}
	\item The \emph{state process} $X=\text{Markov}(\clA,\mu)$
          taking values in the state-space $\bS$ over a finite time horizon $[0,T]$. See Sec.~\ref{ssec:model-X}.
	\item The \emph{observation process} $Z = \{Z_t:0\le t\le T\}$ is $\Re^m$-valued with $Z_0= 0$, and the property that any increment $Z_s-Z_t$ is independent of $X_t$ for all $s>t$. 
	\item The state-observation pair $(X,Z)$ is a $\bS\times
          \Re^m$-valued Feller-Markov process, and a sample path
          $(X,Z)(\omega) \in D\big([0,T];\bS\times\Re^m\big)$ for each
          $\omega\in\Omega$. 
	\item The filtration generated by the observation process is
          denoted by $\{\clZ_t:0\le t\le T\}$ where $\clZ_t :=
          \sigma\big(\{Z_s:0\le s\le t\}\big)$. $\clZ_t\subset \clF_T := \sigma((X,Z))$. The notation $\sigma(\cdot)$ is used to denote the $\sigma$-algebra. 
\end{itemize}

In the study of HMMs, the basic object of interest is the conditional expectation
\[
\pi_t(f) = \E (f(X_t) | \clZ_t),\quad f\in C_b(\bS)
\]
The conditional measure $\pi_t(\cdot)$ is referred to as the {\em nonlinear
  filter}. Without conditioning (e.g., if $\clZ_t$ is trivial), the conditional expectation reduces to the expectation which was our starting point in the first part of this paper. In this case of trivial filtration, $\pi_t$ equals $\mu_t$ -- as it was defined in the first part of this paper. More generally, when measurements carry information regarding the hidden state (i.e., $\clZ_t$ is non-trivial), $\pi_t$ is a random probability measure.  Our first goal is to introduce definitions for divergence for such random probability measures. For this purpose, the following construction is introduced.      

Let $\rho \in \clP(\bS)$.  On the common measurable space $(\Omega,
\clF_T)$, $\sP^\rho$ is used to denote another probability measure
such that the transition law of $(X,Z)$ is identical but $X_0\sim
\rho$ (see~\cite[Sec.~2.2]{clark1999relative} for an explicit construction
of $\sP^\rho$ as a probability measure over $ D\big([0,T];\bS\times\Re^m\big)$.).  The associated expectation operator is denoted by
$\E^\rho(\cdot)$ and the nonlinear filter by $\pi_t^\rho(f) =
\E^\rho\big(f(X_t)|\clZ_t\big)$.  The
two most important choices for $\rho$ are as follows:
\begin{itemize}
\item $\rho=\mu$.
The measure $\mu$ has the meaning of the true prior.  
\item  $\rho=\nu$. 
The measure $\nu$ has
  the meaning of the incorrect prior that is used to compute the
filter.  It is assumed
that $\mu\ll\nu$.
\end{itemize}
The relationship between $\sP^\mu$ and $\sP^\nu$ is as follows ($\sP^\mu|_{\clZ_t}$ denotes
the restriction of $\sP^\mu$ to the $\sigma$-algebra $\clZ_t$):

\medskip

\begin{lemma}[Lemma 2.1 in \cite{clark1999relative}] \label{lm:change-of-Pmu-Pnu}
	Suppose $\mu\ll \nu$. Then 
	\begin{itemize}
		\item $\sP^\mu\ll\sP^\nu$, and the change of measure is given by
		\begin{equation*}\label{eq:P-mu-P-nu}
			\frac{\ud \sP^\mu}{\ud \sP^\nu}(\omega) =
                        \frac{\ud \mu}{\ud
                          \nu}\big(X_0(\omega)\big)\quad
                        \sP^\nu\text{-a.s.} \;\omega
			%\E^\mu\big(\phi(X,Z)\big) = \E^\nu\Big(\phi(X,Z)\frac{\ud \mu}{\ud \nu}(X_0)\Big)
		\end{equation*}
		%for any given functional $\phi:C\big([0,T];\bS\big)\times C\big([0,T];\Re^m\big)\to \Re$.
		\item For each $t > 0$, $\pi_t^\mu \ll \pi_t^\nu$, $\sP^\mu|_{\clZ_t}$-a.s..
	\end{itemize} 
	
\end{lemma}

\medskip

The following definition of filter stability is based on
$f$-divergence (Because $\mu$ has the meaning of the correct prior,
the expectation is with respect to $\sP^\mu$):

\medskip

\begin{definition}\label{def:filter-stability}
The nonlinear filter is said to be \emph{stable} in the sense of 
	\begin{align*}
		\text{(KL divergence)\quad if}\qquad& \E^\mu\big(\kl(\pi_T^\mu\mid \pi_T^\nu)\big) \; \longrightarrow\; 0\\
		\text{($\chi^2$ divergence)\quad if}\qquad& \E^\mu\big(\chisq(\pi_T^\mu\mid \pi_T^\nu)\big) \; \longrightarrow\; 0\\
		\text{(Total variation)\quad if}\qquad& \E^\mu\big(
                                                \|\pi_T^\mu - \pi_T^\nu\|_\tv\big) \; \longrightarrow\; 0
	\end{align*}
	as $T\to \infty$ for every $\mu, \nu\in\clP(\bS)$ such that
        $\mu\ll \nu$.
\end{definition}

\medskip

Apart from $f$-divergence based definitions, the following definitions of 
filter stability are also of historical interest:

\medskip

\begin{definition}\label{def:FS}
The nonlinear filter is said to be stable in the sense of
	\begin{align*}
		\text{($L^2$)\quad if}\qquad & \E^\mu\big(|\pi_T^\mu(f)-\pi_T^\nu(f)|^2\big) \;\longrightarrow\; 0\\
		\text{(a.s.)\quad if}\qquad & |\pi_T^\mu(f) - \pi_T^\nu(f)|\;\longrightarrow\; 0\quad \sP^\mu\text{-a.s.}
	\end{align*}
	as $T\to \infty$, for every $f\in C_b(\bS)$ and $\mu,\nu\in\clP(\bS)$ s.t.~$\mu \ll \nu$.
\end{definition}

\medskip

Our own prior work has been to show filter stability in the
sense of $\chisq$-divergence.  Based on well known relationship between
f-divergences, this also implies other types of  stability as follows:

\medskip

\begin{proposition}\label{prop:chisq-stability-implication}
	If the filter is stable in the sense of $\chisq$ then it is stable in KL divergence, total variation, and $L^2$.
\end{proposition}

\begin{proof}
See Appendix~\ref{apdx:pf-chisq}.
\end{proof}

\subsection{Analysis of the filter using KL divergence}

An early work appears in~\cite{clark1999relative} whose main result is as follows:

\begin{proposition}[Theorem 2.2 in~\cite{clark1999relative}] \label{prop:clark-main-result}
  Assume $\mu\ll \nu$. Then
  \begin{equation}\label{eq:KL-Lyapunov}
	\E^\mu\big(\kl(\pi_T^\mu\mid\pi_T^\nu)\big) \le \kl(\mu\mid\nu),\quad \forall \;T\geq 0
\end{equation}
\end{proposition}

\medskip

The result in~\cite[Thm.~2.2]{clark1999relative} is given in a much stronger form, stated as an equality (which will require additional notation to state here).  The inequality suffices for our purpose. The inequality implies that
$\big\{\kl(\pi_t^\mu\mid\pi_t^\nu):t\ge 0\big\}$ is  a $\sP^\mu$-super-martingale. This is because upon conditioning on $\clZ_s$ for $s \le t$ and using~\eqref{eq:KL-Lyapunov} now with initializations $\pi_s^\mu$ and $\pi_s^\nu$ at time $s$,
\[
\E^\mu\big(\kl(\pi_t^\mu\mid\pi_t^\nu)\mid \clZ_s\big)  \le \kl(\pi_s^\mu\mid\pi_s^\nu),\quad 0\le s\le t,\;\sP^\mu\text{-a.s.} 
\]
Therefore, the K-L divergence is a candidate Lyapunov function for the filter, in the sense that $\E^\mu\big(\kl(\pi_t^\mu\mid\pi_t^\nu)\big)$ is non-increasing as a function of time $t$. However, it is generally regarded as a difficult problem to use these to show that, under appropriate conditions on the HMM, $\E^\mu\big(\kl(\pi_t^\mu\mid\pi_t^\nu)\big) \to 0$ as $t\to \infty$~\cite[Section 4.1]{chigansky2009intrinsic}.

\subsection{Analysis of the filter using total variation}
This section is adapted from papers on the intrinsic approach to the problem of filter stability~\cite{chigansky2009intrinsic}.  We begin with recalling the Bayes formula,
\[
\E^\mu\big(f(X_T)\mid \clZ_T\big) = \frac{\E^\nu\big(\frac{\ud \sP^\mu}{\ud \sP^\nu}f(X_T)\mid \clZ_T\big)}{\E^\nu\big(\frac{\ud \sP^\mu}{\ud \sP^\nu}\mid \clZ_T\big)}
\]
Since $\frac{\ud \sP^\mu}{\ud \sP^\nu} = \frac{\ud \mu}{\ud \nu}(X_0)$ (Lemma~\ref{lm:change-of-Pmu-Pnu}), the equation is expressed as follows:
\begin{align*}
	\int_\bS f(x)\ud \pi_T^\mu(x) &= \frac{\E^\nu\big(\frac{\ud \mu}{\ud \nu}(X_0)f(X_T)\mid \clZ_T\big)}{\E^\nu\big(\frac{\ud \mu}{\ud \nu}(X_0)\mid \clZ_T\big)}\\
	&= \E^\nu\Big(f(X_T) \frac{\E^\nu\big(\frac{\ud \mu}{\ud \nu}(X_0)\mid \clZ_T \vee \sigma\{X_T\}\big)}{\E^\nu\big(\frac{\ud \mu}{\ud \nu}(X_0)\mid \clZ_T\big)}\Big)\\
	&= \int_\bS f(x) \frac{\E^\nu\big(\frac{\ud \mu}{\ud \nu}(X_0)\mid \clZ_T, X_T = x\big)}{\E^\nu\big(\frac{\ud \mu}{\ud \nu}(X_0)\mid \clZ_T\big)} \ud \pi_T^\nu(x)
\end{align*}
This calculation shows that
\begin{equation*}\label{eq:gamma_T-explicit}
	 \frac{\ud \pi_T^\mu}{\ud \pi_T^\nu}(x) =  \frac{\E^\nu\big(\frac{\ud \mu}{\ud \nu}(X_0)\mid \clZ_T, X_T = x\big)}{\E^\nu\big(\frac{\ud \mu}{\ud \nu}(X_0)\mid \clZ_T\big)},\quad x\in\bS,\quad \sP^\mu\text{-a.s.}
\end{equation*}
The identity is the starting point of the intrinsic approach to the problem of filter stability.  The identity is used to define the \emph{forward map} $\gamma_0\mapsto \gamma_T$ as follows:
\begin{equation*}\label{eq:fmap}
	\gamma_T(x) = \E^\nu
	\Big( \frac{\gamma_0(X_0)}{\E^\nu(\gamma_0(X_0) \mid \clZ_T)}
        \,\bigg|\,\clZ_T\vee [X_T=x]\Big),\quad x\in \bS
\end{equation*}
where $\gamma_T:=\frac{\ud \pi_T^\mu}{\ud \pi_T^\nu}$ and $\gamma_0:=\frac{\ud \pi_T^\mu}{\ud \pi_T^\nu}$.  Based on the forward map, 
\begin{align*}
	\|\pi_T^\mu-\pi_T^\nu\|_\tv &= \int_\bS \big|\gamma_T(x)-1\big| \ud \pi_T^\nu(x)\\
	&=\frac{\E^\nu\Big(\big|\E^\nu\big(\gamma_0(X_0)\mid \clZ_T\vee \sigma\{X_T\}\big) - \E^\nu\big(\gamma_0(X_0)\mid \clZ_T\big)\big| \mid \clZ_T\Big)}{\E^\nu\big(\gamma_0(X_0)\mid \clZ_T\big)}
\end{align*}
Since $(X,Z)$ is a Markov process, $X_0$ and $\sigma(\{(X_t,Z_t):t\geq T\})$ are conditionally independent given $(X_T,Z_T)$. Therefore, the first term in the numerator can be expressed as
\[
\E^\nu\big(\gamma_0(X_0)\mid \clZ_T\vee \sigma\{X_T\}\big) = \E^\nu\big(\gamma_0(X_0)\mid \clZ_\infty\vee \clF^X_{[T,\infty)}\big) 
\]
where $\clZ_\infty=\bigcup_{T>0} \clZ_T$ and $\clF^X_{[T,\infty)} = \sigma\big(\{X_t: t \ge T\}\big)$ is the tail sigma algebra of the state process $X$.  Hence,
\begin{align*}
  \E^\nu\Big(\gamma_0(X_0) &\,\big|\, \clZ_T\Big)\|\pi_T^\mu-\pi_T^\nu\|_\tv \\
  & = \E^\nu\Big(\Big|\E^\nu\big(\gamma_0(X_0)\mid \clZ_\infty\vee \clF^X_{[T,\infty)}\big) - \E^\nu\big(\gamma_0(X_0)\mid \clZ_T\big)\Big|\,\big|\, \clZ_T\Big)
\end{align*}
Taking expectation $\E^\nu(\cdot)$ on both sides, noting $\frac{\ud \sP^\mu}{\ud \sP^\nu} = \gamma_0(X_0)$, yields 
\[
\E^\mu\big(\|\pi_T^\mu-\pi_T^\nu\|_\tv\big) = \E^\nu\Big(\,\Big|\E^\nu\big(\gamma_0(X_0)\mid \clZ_\infty\vee \clF^X_{[T,\infty)}\big) - \E^\nu\big(\gamma_0(X_0)\mid \clZ_T\big)\Big|\,\Big)
\]
Upon taking an expectation,
\[
\E^\mu\big(\|\pi_T^\mu-\pi_T^\nu\|_\tv\big) = \E^\nu\Big(\,\Big|\E^\nu\big(\gamma_0 (X_0)\mid \clZ_\infty\vee \clF^X_{[T,\infty)}\big) - \E^\nu\big(\gamma_0(X_0)\mid \clZ_T\big)\Big|\,\Big)
\]
Now, as a function of $T$, $\clZ_\infty\vee \clF^X_{[T,\infty)}$ is a decreasing filtration and $\clZ_T$ is an increasing filtration. Therefore, 
\begin{align*}
\lim_{T\to \infty} & \E^\mu\big(\|\pi_T^\mu-\pi_T^\nu\|_\tv\big) = \\
&\E^\nu\Big(\,\Big|\E^\nu\big(\gamma_0 (X_0)\mid \bigcap_{T\ge 0}\clZ_\infty\vee \clF^X_{[T,\infty)}\big) - \E^\nu\big(\gamma_0 (X_0)\mid \clZ_\infty\big)\Big|\, \Big)
\end{align*}
The right-hand side is zero if the following tail sigma field identity is satisfied:
\begin{equation*}\label{eq:kunita-tail-sigmafield}
	\bigcap_{T\ge 0} \clZ_\infty \vee \clF^X_{[T,\infty)} \stackrel{?}{=} \clZ_\infty %\vee \bigcap_{T\ge 0} \clF^X_{[T,\infty)}
\end{equation*}
This identity is referred to as the central problem in the stability analysis of the nonlinear filter~\cite{van2010nonlinear}. The problem generated large consequent attention (see~\cite{budhiraja2003asymptotic} and references therein).

\subsection{Analysis of the filter using  $\chi^2$-divergence}\label{sec:bmap_section}

The starting point of our own work on the topic of filter stability is the \emph{backward
  map} $\gamma_T\mapsto y_0$ defined as follows:
\begin{equation}\label{eq:bmap}
	y_0(x) :=  \E^\nu (\gamma_T(X_T)|[X_0=x]),\quad x\in\bS
\end{equation}
(we adopt here the convention that $\frac{0}{0}=0$). 
The function $y_0:\bS\to \Re$ is deterministic, non-negative, and 
$\nu(y_0) =  \E^\nu (\gamma_T(X_T)) = 1$, and therefore is also a likelihood ratio.

Since $\mu\ll\nu$, it follows $\mu(y_0) = \E^\mu\big(\gamma_T(X_T)\big)$. Using the tower property, 
\[
\mu(y_0) = \E^\mu\big(\gamma_T(X_T)\big) = \E^\mu\big(\pi_T^\mu(\gamma_T)\big) = \E^\mu\big(\pi_T^\nu(\gamma_T^2)\big)
\]
Noting that $\pi_T^\nu(\gamma_T^2)-1 = \chisq\big(\pi_T^\mu\mid\pi_T^\nu\big)$ is the $\chisq$-divergence, 
\[
\E^\mu\big(\chisq\big(\pi_T^\mu\mid\pi_T^\nu\big)\big) = \mu(y_0) -
\nu(y_0)
\]
Therefore, filter stability in the sense of $\chisq$-divergence is
equivalent to showing that $\mu(y_0) \stackrel{(T\to
  \infty)}{\longrightarrow} 1$. Because 
$
\mu(y_0)-\nu(y_0) = \nu\big((\gamma_0-1)(y_0-1)\big)
$, 
upon using the Cauchy-Schwarz inequality, 
\begin{equation}\label{eq:chisq_inequality}
	\E^\mu\big(\chi^2(\pi_T^\mu\mid\pi_T^\nu)\big)^2\leq \var^\nu(y_0(X_0)) \; \chi^2(\mu \mid \nu)
\end{equation}
where $\var^\nu(y_0(X_0))=\E^\nu\big( |y_0(X_0)-1|^2\big)$. Consequently, provided $\chi^2(\mu \mid \nu) <\infty$, filter stability follows from the following:
\begin{align*}\label{eq:VDP}
	&\textbf{(variance decay prop.)}\qquad  \var^\nu(y_0(X_0))  \stackrel{(T\to\infty)}{\longrightarrow}
	0 &
% \;\;\forall \mu\ll \nu&
\end{align*}
Next, from the definition of the backward map,
\[
(y_0(X_0) -1) =  \E^\nu ((\gamma_T(X_T)-1)|X_0),\quad \sP^\nu-\text{a.s.}
\]
and using
Jensen's inequality, 
\begin{equation}\label{eq:jensens_ineq}
	\var^\nu(y_0(X_0)) \le \var^\nu(\gamma_T(X_T))
\end{equation}
where $\var^\nu(\gamma_T(X_T)):=\E^\nu\big(|\gamma_T(X_T)-1|^2\big)$. 
Therefore, the backward map $\gamma_T\mapsto y_0$ is non-expansive.  That is, the variance of the random variable $y_0(X_0)$ is smaller than the
variance of the random variable $\gamma_T(X_T)$.  One might expect
that the 
variance decay property follows from showing that the backward map is
a strict contraction.  Based on the formulae for the Markov processes,
one might then also 
expect the (exponential) rate of contraction to be the suitably
defined generalization of the Poincar\'e
constant.       

The above program is being carried out for the HMM with white noise
observations~\cite{kim2024variance,kim2021ergodic,kim2021detectable,JinPhDthesis}.
A summary of the some results using this approach appear in
the~\Sec{sec:bmap_opt} after the model has been introduced.

\subsection{Application to stochastic thermodynamics}\label{sec:stoch_thermo_HMM}

In the presence of observations, the statement of the second law of thermodynamics becomes more nuanced. In 1867, physicist James Maxwell introduced the concept of Maxwell's demon, a hypothetical being that uses information to transform thermal energy into work, seemingly violating the second law of thermodynamics.
In \Sec{sec:maxwell_demon}, Maxwell's demon is illustrated with the aid of the over-damped Langevin HMM with white noise observations.

\section{Hidden Markov model with white noise observations}\label{sec:white-noise}

\subsection{Model and math preliminaries}

In this section, a special class of HMM 
is considered as follows:
\begin{itemize}
	\item $X=\text{Markov}(\clA,\mu)$ taking values in the state-space $\bS$. 
	\item $Z$ is given by the SDE
	\begin{equation} \label{eq:obs-model}
	Z_t = \int_0^t h(X_s)\ud s + W_t,\quad t\ge 0
	\end{equation}
	where the function $h:\bS\to \Re^m$ is referred to as the observation function and $W = \{W_t:0\le t\le T\}$ is an $m$-dimensional Brownian motion. It is assumed that $W$ is independent of $X$. 
\end{itemize}
The above is often referred to as the \emph{white noise observation model} of nonlinear filtering. The model is succinctly denoted by $(\clA, h)$.

Assuming a certain technical (Novikov) condition holds, the nonlinear filter solves the \emph{Kushner-Stratonovich equation}: 
\begin{equation}\label{eq:Kushner}
	\ud \pi_t(f) = \pi_t(\clA f) \ud t +
	\big(\pi_t(hf)-\pi_t(f)\pi_t(h)\big)^\tp \ud I_t
\end{equation}
where the \emph{innovation process} is defined by
\[
I_t := Z_t - \int_0^t \pi_s(h)\ud s,\quad t \ge 0
\]

The two conditional measures of interest, $\{\pi_t^\mu : t\geq 0\}$ and $\{\pi_t^\nu : t\geq 0\}$, are solutions of~\eqref{eq:Kushner}, starting from initialization $\pi_0=\mu$ and $\pi_0=\nu$, respectively. The respective innovations processes are denoted as $\{I_t^\mu : t\geq 0\}$ and $\{I_t^\nu : t\geq 0\}$.  From filtering theory, these are known to be Wiener Processes, with respect to $\sP^\mu$ and $\sP^\nu$, respectively. 

\subsection{Formulae for $f$-divergence}

For the HMM $(\clA,h)$, explicit formulae for the
              two types of f-divergences appear in the following
              theorem.  (The total variation is not $C^1$ in its
              argument and therefore not amenable to taking the
              derivative). 
\begin{theorem}\label{thm:div_filter}
\begin{align} 
\ud\,\kl(\pi_t^\mu\mid\pi_t^\nu) &= -\frac{1}{2}\pi_t^\nu\Big(\frac{\Gamma \gamma_t}{\gamma_t}\Big) \ud t - \frac{1}{2}|\pi_t^\mu(h)-\pi_t^\nu(h)|^2\ud t \label{eq:rate-kl-Markov} \\
&\qquad + \big[\pi_t^\mu\big(\log\gamma_t (h-\pi_t^\mu(h))\big) - \big(\pi_t^\mu(h) - \pi_t^\nu(h)\big)\big]^\tp\ud I_t^\mu  \nonumber \\
\ud\,\chisq(\pi_t^\mu\mid\pi_t^\nu) &=-\pi_t^\nu\big(\Gamma
                                      \gamma_t\big) \ud t 
- \clV_t^\mu (\gamma_t,h)^\tp \clV_t^\nu (\gamma_t,h) \ud t
\label{eq:rate-chisq-Markov} \\
&\qquad + \pi_t^\mu\big(\gamma_t(h-\pi_t^\mu)-\gamma_t(\pi_t^\mu(h)-\pi_t^\nu(h))\big)^\tp\ud I_t^\mu \nonumber
\end{align}
\end{theorem}
where $\clV_t^\mu (\gamma_t,h) = \pi_t^\mu(\gamma_t h) -
\pi_t^\mu(\gamma_t)\pi_t^\mu(h)$ and $\clV_t^\nu (\gamma_t,h) = \pi_t^\nu(\gamma_t h) -
\pi_t^\nu(\gamma_t)\pi_t^\nu(h)$. 
\begin{proof}
See Appendix~\ref{apdx:pf-thm-div-filter}.
\end{proof}

\subsection{Difficulties with application to filter stability}

Parallel to the discussion for stochastic stability in \Sec{sec:MP}, one may seek to use the formulae in \Thm{thm:div_filter} to extract conclusions on filter stability.  Note that the first term of each of the two equations, 
$-\frac{1}{2}\pi_t^\nu\Big(\frac{\Gamma \gamma_t}{\gamma_t}\Big)$ in~\eqref{eq:rate-kl-Markov} 
and $-\pi_t^\nu\big(\Gamma\gamma_t\big)$
in~\eqref{eq:rate-chisq-Markov}, has an analogue in
Thm.~\ref{thm:div_MP}.  Both of these are non-positive terms that capture the effect due to the
dynamics. 

\medskip

\noindent \textbf{Filter stability using KL-divergence.} Upon taking
an expectation of both sides, noting that $\{I_t^\mu:t\geq 0\}$ is a
$\sP^\mu$-martingale, 
\[
\frac{\ud}{\ud t} \E^\mu \left(\kl(\pi_t^\mu\mid\pi_t^\nu) \right) =
-\frac{1}{2}\E^\mu \left(\pi_t^\nu\Big(\frac{\Gamma
    \gamma_t}{\gamma_t}\Big) \right) -\frac{1}{2}\E^\mu \left(|\pi_t^\mu(h)-\pi_t^\nu(h)|^2\right)
\]
Each of the two terms on the right-hand side is non-positive which
shows that $\{\kl(\pi_t^\mu\mid\pi_t^\nu):t\geq 0\}$ is a
$\sP^\mu$-super-martingale.  (This is consistent with the conclusion
obtained previously from \Prop{prop:clark-main-result}).  However, the formula has
not been very useful to obtain meaningful insights on model properties
(for $(\clA,h)$) 
such that the divergence converges to zero.    

\medskip

\noindent \textbf{Analysis using $\chi^2$-divergence.} The term $-\pi_t^\nu(\Gamma \gamma_t)$ on the right-hand
side of~\eqref{eq:rate-chisq-Markov} is non-positive. However, the
second term $\clV_t^\mu (\gamma_t,h)^\tp \clV_t^\nu (\gamma_t,h)$ can be
either positive or negative. Furthermore, the second term scales with $h$ and
therefore it may become arbitrarily large.  Therefore,
        the equation has not been useful for 
        the asymptotic analysis of the $\chisq$-divergence.

\medskip

The analysis of filter stability in the sense of $\chi^2$-divergence
is made possible
based on the use of the backward map~\eqref{eq:bmap} introduced
in~\Sec{sec:bmap_section}..  This is our original work which is
summarized in the following section.

\subsection{Backward map and the analysis of $\chisq$-divergence}\label{sec:bmap_opt}

In this section, we continue the analysis of the backward
map $\gamma_T \mapsto y_0$ introduced as equation~\eqref{eq:bmap} 
in~\Sec{sec:bmap_section}.  For this purpose, consider the backward
stochastic differential equation (BSDE):
\begin{align}
	-\ud Y_t(x) &= \big((\clA Y_t)(x) + h^\tp(x)V_t(x)\big)\ud t -
                      V_t^\tp(x) \ud Z_t, \nonumber\\
	\quad Y_T(x) &= \gamma_T(x),\;\; x\in\bS, \;\;0\leq t\leq T \label{eq:closed-loop-bsde-2}
\end{align}
Here $(Y,V) = \{(Y_t(x),V_t(x)):\Omega \to \Re\times \Re^m
\;:\; x\in\bS, \;0\leq t\leq T\}$ is a $\clZ$-adapted solution of the
BSDE for a prescribed $\clZ_T$-measurable terminal condition
$Y_T=\gamma_T=\frac{\ud 
		\pi^\mu_T}{\ud \pi^\nu_T}$.  (There is a well established existence,
uniqueness, and regularity theory for the BSDE;
cf.,~\cite{Zhang2017}.) For additional motivation concerning the BSDE,
the reader is referred to~\cite{JinPhDthesis,duality_jrnl_paper_I,duality_jrnl_paper_II}.

For the HMM $(\clA,h)$, the relationship between the BSDE~\eqref{eq:closed-loop-bsde-2} and the backward map~\eqref{eq:bmap} is given
in the following proposition:

\begin{proposition}\label{prop:backward-map-and-bsde}
Consider~\eqref{eq:closed-loop-bsde-2}.  Then at
        time $t=0$,
\begin{equation*}\label{eq:bmap-and-bsde}
Y_0(x)=y_0(x) ,\quad x \in \bS
\end{equation*}
where $y_0$ is according to the backward map~\eqref{eq:bmap}.  Also,
\begin{equation}\label{eq:var_contractive_ext}
\frac{\ud }{\ud t} \var^\nu (Y_t(X_t)) = \E^\nu \big(
  \pi_t^\nu(\Gamma Y_t) +
  \pi_t^\nu(|V_t|^2) \big),\quad 0\leq t\leq T
\end{equation}
where $\var^\nu(Y_t(X_t)) :=\E^\nu(|Y_t(X_t)-1|^2)$ and $
\pi_t^\nu(|V_t|^2):=\int_\bS V^\tp_t(x) V_t(x)\ud \pi_T^\nu(x) $. 
\end{proposition} 

\begin{proof}
See Appendix~\ref{apdx:pf-prop-bmap-bsde}.
\end{proof}
 
\begin{remark}[Relationship to~\eqref{eq:jensens_ineq}]
Upon integrating~\eqref{eq:var_contractive_ext},
\begin{align} 
\var^\nu(Y_0(X_0)) & + \E^\nu \Big(\int_0^T \pi_t^\nu(\Gamma Y_t) +
  \pi_t^\nu(|V_t|^2) \ud t \Big) = \var^\nu(\gamma_T(X_T)) \label{eq:contraction_bmap2}
\end{align}
which is the stronger form of the inequality
$\var^\nu(y_0(X_0))\leq \var^\nu(\gamma_T(X_T))$ 
first obtained (see~\eqref{eq:jensens_ineq}) using the Jensen's inequality. 
\end{remark}

Either of the two equations,~\eqref{eq:var_contractive_ext}
or~\eqref{eq:contraction_bmap2}, is useful to show the variance decay
property under suitable additional conditions on the model.  The most
straightforward of these conditions is the conditional Poincar\'e
inequality (c-PI) introduced in an early paper on this topic:
	\begin{equation*}\label{eq:c-PI}
		\text{(c-PI)}\qquad\qquad	\pi_t^\bmu (\Gamma f) \geq c \;
		 \pi_t^\bmu (|f- \pi_t^\bmu(f)|^2)\quad \forall\, f \in\clD,\; t\geq 0
	\end{equation*}

Based upon c-PI, the following result is shown:

\begin{proposition}\label{prop:functional-inequality}
Suppose c-PI holds with constant $c$.  Let $\underline{a} :=
	\mathop{\operatorname{essinf}}_{x\in\bS} \frac{\ud \mu}{\ud \bmu} (x)$. Then
	\[
	\E^\mu\big(\chisq(\pi_T^\mu\mid\pi_T^\bmu)\big) \le
	\frac{1}{\underline{a}} e^{-cT} \; \chisq(\mu\mid\bmu)
	\]
\end{proposition}
\begin{proof}
See Appendix~\ref{apdx:pf-filter-stability}.
\end{proof}

\begin{remark}
In the conference paper~\cite{kim2021ergodic}, a few cases of HMM are
described where the c-PI can be verified.  This includes the case
where the Markov process is Doeblin.  In a more recent paper~\cite{kim2024variance}, a weaker
definition of the Poincar\'e inequality is introduced and related to
the HMM model properties, namely, observability of $(\clA,h)$ and
ergodicity of $\clA$. The definition is used to show filter
stability under conditions weaker than the c-PI.  The reader is
referred to~\cite{kim2024variance} for results in this direction. 
\end{remark}

\subsection{Application to stochastic thermodynamics}\label{sec:maxwell_demon}

For the purposes of illustrating the Maxwell's demon, consider the
time-varying over-damped Langevin model~\eqref{eq:dyn_sde_TV}
together with observations according to~\eqref{eq:obs-model}.  The
resulting 
Langevin HMM
is as follows:
\begin{subequations}\label{eq:Langevin-obs-model}
	\begin{align}
		\ud X_t &= -\nabla U_t (X_t) \ud t + \sqrt{2}\ud B_t, \quad
		X_0\sim \mu \\
		\ud Z_t  &= h(X_t) \ud t + \ud W_t, \quad Z_0=0
	\end{align}
\end{subequations}
where $X_0,B,W$ are assumed to be mutually independent.
The distinction from the time-varying model~\eqref{eq:dyn_sde_TV} (as discussed in \Sec{sec:2nd-law}) lies in the fact that the potential $U_t(\cdot)$ is now allowed to be a $\clZ_t$-measurable function. In other words, the demon can select the potential function at time based on observations made up to that time. In contrast, in \Sec{sec:2nd-law}, the potential function is time-varying but deterministic.

To introduce a modified version of the second law (particularly equation~\eqref{eq:2nd-law-ineq}) that accounts for the demon’s ability to select a potential based on prior observations, we consider the nonlinear filter $\{\pi^\mu_t:t\geq 0\}$ and use it to define the following thermodynamic quantities:  
\begin{align*}
\text{(conditional free energy):} \qquad \qquad
	\Delta \mathcal F& = \mathcal F(\pi^\mu_T, U_T) -  \mathcal F(\mu, U_0)  \\
	\text{(conditional work):} \qquad \qquad	\mathcal W_{T} &:= \int_0^{T} \pi_t^\mu(\frac{\partial}{\partial t} U_t) \ud t\\ 
	\text{(dissipation):}\qquad\qquad\quad \cal{D} & :=  \E^\mu \left(
	\int_0^{T} \fish(\pi_t^\mu \mid \nu_t  ) \ud t \right) \\
	\text{(mutual information):}\qquad 
	\mathcal I(X,Z) & =\frac{1}{2}\E^\mu \left(\int_0^{T} \mathsf |h(X_t)
	-\pi_t^\mu(h)|^2 \ud  t \right)
\end{align*} 
where, as before, $\nu_t$ is the Boltzmann measure defined based on $U_t$ (That is, $\frac{\ud \nu_t}{\ud x}(x) \propto e^{-U_t(x)}$.) 
The main result of~\cite{taghvaei2021relation} is as follows:
\begin{theorem}\label{thm:2nd-law-obs}
	Consider the over-damped Langevin HMM~\eqref{eq:Langevin-obs-model}. Then, 
	\begin{align}\label{eq:2nd-law-with-obs}
		 \E^\mu ( 	\mathcal W_{T}  - \Delta \mathcal F ) = \mathcal D - \mathcal I(X,Z)  
	\end{align}
\end{theorem}
\begin{proof}
See Appendix~\ref{apdx:pf-2nd-law-obs}.
\end{proof}
Compared to~\eqref{eq:2nd-law-ineq}, the equation~\eqref{eq:2nd-law-with-obs} involves an additional term because of
mutual information $ \mathcal I(X,Z)$.  While the dissipation $\clD$
is positive, the additional term may cause $\mathsf \E^\mu (
\mathcal W_{T}  - \Delta \mathcal F )$ to become negative. Derivation
of similar types of equalities and inequalities, under different
settings of the model, has been subject of recent research in the stochastic thermodynamic literature~\cite{sagawa2008discreteqmeas,Horowitz_2014,parrondo2021extracting,taghvaei2021relation}. 

\subsection{Acknowledgement}

This work is supported in part by the AFOSR award FA9550-23-1-0060 and the NSF award 2336137 (Mehta), the NSF award EPCN-2318977 (Taghvaei), and the DFG grant 318763901/SFB1294 (Kim).  A portion of this
work was conducted during authors' (Taghvaei and Mehta) visit to the
Banff International Research Station as part of the ``Geometry, Topology
and Control System Design'' workshop in Banff from June 11 to June 16, 2023.

\appendix

\section{Proofs}

\subsection{Proof of Theorem~\ref{thm:div_MP}}\label{apdx:pf-thm-div-MP}

\newP{Computation for KL divergence}
%Using the den
\[
\kl\big(\mu_t\mid\nu_t\big) = \int_\bS \ud \mu_t(x)\log \mu_t(x) - \ud
\mu_t(x) \log \nu_t(x)
\]
and therefore by applying chain rule,
\begin{align*}
	\frac{\ud }{\ud t}\kl(\mu_t\mid \nu_t) &= \int_\bS \ud
                                                 \mu_t(x)
                                                 \clA\big[\log
                                                 \frac{\mu_t}{\nu_t} +
                                                 1\big] - \int \ud \nu_t(x)\clA\big[\frac{\mu_t}{\nu_t}\big] \\
	&=\int_\bS \ud \nu_t(x) \Big[\frac{\mu_t(x)}{\nu_t(x)}\clA\big[\log \frac{\mu_t}{\nu_t}\big] -\clA\big[\frac{\mu_t}{\nu_t}\big]\Big] \\
	&= \nu_t\big(\gamma_t \clA \log \gamma_t - \clA \gamma_t\big)
\end{align*}
It remains to show that $\gamma_t \clA \log \gamma_t - \clA \gamma_t = \frac{\Gamma \gamma_t}{2\gamma_t}$ at each $t$. First we introduce a martingale problem for the generator
\[
N_t^f := f(X_t) - \int_0^t (\clA f) (X_s)\ud s \quad \text{is a martingale}
\]
and its quadratic variation is $(\Gamma f)(X_t)\ud t$. 
For a fixed positive function $\gamma(x)$, applying It\^o rule on the derivative of $\log \gamma(X_t)$,
\begin{align*}
	\ud \log \gamma(X_t) &= \frac{1}{\gamma(X_t)}\ud \gamma(X_t) - \frac{1}{2\gamma^2(X_t)}(\Gamma \gamma)(X_t)\ud t\\
	&= \frac{1}{\gamma(X_t)}\ud N_t^\gamma + \frac{1}{\gamma(X_t)}(\clA \gamma)(X_t)\ud t - \frac{1}{2\gamma^2(X_t)}(\Gamma \gamma)(X_t)\ud t
\end{align*}
Therefore, we obtain $\clA \log \gamma$ by grouping all finite variation terms:
\[
\clA (\log \gamma) = \frac{1}{\gamma}\clA \gamma - \frac{1}{2\gamma^2}(\Gamma \gamma)
\]

\newP{Computation for $\chisq$ divergence}
\[
\chisq(\mu_t\mid \nu_t) = \int \frac{\mu_t(x)}{\nu_t(x)} \ud \mu_t(x) -1
\]
and therefore
\begin{align*}
\frac{\ud}{\ud t} \chisq(\mu_t\mid \nu_t) &= \int \ud \mu_t(x)
                                            \clA\big[\frac{2\mu_t}{\nu_t}\big] - \int \ud \nu_t(x)\clA\big[\frac{\mu_t^2}{\nu_t^2}\big]\\
&= -\nu_t\big(\clA \gamma_t^2 - 2\gamma_t \clA \gamma_t\big) = -\nu_t\big(\Gamma \gamma_t\big)
\end{align*}
Hence the formula is obtained.

\subsection{Proof of Proposition~\ref{prop:chisq-stability-implication}}\label{apdx:pf-chisq}

For considered $f$-divergences, the following inequalities are standard (see~\cite[Lemma 2.5 and 2.7]{Tsybakov2009estimation}):
\[
2\|\mu-\nu\|_\tv^2 \le \kl(\mu\mid \nu)\le \chisq(\mu\mid\nu)
\]
The first inequality is called the Pinsker's inequality.  The result
follows directly from using these inequalities.  For $L^2$ stability, observe that for any $f \in C_b(\bS)$,
\[
\pi_T^\mu(f) - \pi_T^\nu(f) = \pi_T^\nu(f\gamma_T) - 
\pi_T^\nu(f)\pi_T^\nu(\gamma_T)
\]
Therefore by Cauchy-Schwarz inequality,
\[
|\pi_T^\mu(f)-\pi_T^\nu(f)|^2 \le \frac{\operatorname{osc}(f)}{4} 
\chisq(\pi_T^\mu\mid \pi_T^\nu)
\]
where $\operatorname{osc}(f) = \sup_{x\in\bS} f(x) - \inf_{x\in\bS} f(x)$ denotes the oscillation of 
$f$. Taking $\E^\mu(\cdot)$ on both sides yields the conclusion.

\subsection{Proof of Theorem~\ref{thm:div_filter}}\label{apdx:pf-thm-div-filter}

Similar to the proof of Thm.~\ref{thm:div_MP}, we assume the probability density functions for $\pi_t^\mu$ and $\pi_t^\nu$.
The result directly follows from It\^o rule: 
\begin{align*}
	\ud \kl(\pi_t^\mu\mid\pi_t^\nu) &= \pi_t^\nu\big(\gamma_t \clA \log \gamma_t - \clA \gamma_t\big) \ud t + \pi_t^\mu\big((\log \gamma_t + 1) (h - \pi_t^\mu(h))\big) \ud I_t^\mu \\
	&\quad - \pi_t^\nu\big(\gamma_t(h - \pi_t^\nu(h))\big)\ud I_t^\nu + \frac{1}{2}\pi_t^\mu\big((h-\pi_t^\mu(h))^2\big)\ud t\\
	&\quad -\pi_t^\mu\big((h-\pi_t^\mu(h))(h-\pi_t^\nu(h))\big)\ud t + \frac{1}{2}\pi_t^\mu\big((h-\pi_t^\nu(h))^2\big)\ud t\\
	&= \pi_t^\nu\big(\gamma_t \clA \log \gamma_t - \clA \gamma_t\big) \ud t - \frac{1}{2}\big(\pi_t^\mu(h)-\pi_t^\nu(h)\big)^2\ud t + \big[\cdots\big]\ud I_t^\mu 
\end{align*}
For the last part we use $\ud I_t^\nu = \ud Z_t - \pi_t^\nu(h)\ud t = \ud I_t^\mu + \big(\pi_t^\mu(h)-\pi_t^\nu(h)\big)\ud t$.

For $\chisq$divergence,
\begin{align*}
	\ud \chisq(\pi_t^\mu\mid \pi_t^\nu) &= -\pi_t^\nu\big(\Gamma \gamma_t\big) \ud t + \pi_t^\mu\big(2\gamma_t(h-\pi_t^\mu(h))\big)\ud I_t^\mu \\
	&\quad -\pi_t^\nu\big(\gamma_t^2(h-\pi_t^\nu(h))\big)\ud I_t^\nu  +\pi_t^\mu\big(\gamma_t(h-\pi_t^\mu(h))^2\big)\ud t \\
	&\quad -2\pi_t^\mu\big(\gamma_t(h-\pi_t^\mu(h))(h-\pi_t^\nu(h))\big)\ud t + \pi_t^\mu\big(\gamma_t(h-\pi_t^\nu(h))^2\big)\ud t\\
	&=\pi_t^\nu\big(\Gamma \gamma_t\big) \ud t + \pi_t^\mu(\gamma_t)\big(\pi_t^\mu(h)-\pi_t^\nu(h)\big)^2\ud t \\
	&\quad -\pi_t^\mu\big(\gamma_t(h-\pi_t^\nu(h))\big)\big(\pi_t^\mu(h)-\pi_t^\nu(h)\big)\ud t +  \big[\cdots\big]\ud I_t^\mu \\
	&=\pi_t^\nu\big(\Gamma \gamma_t\big) \ud t -  \big(\pi_t^\mu(\gamma_th)-\pi_t^\mu(\gamma_t)\pi_t^\mu(h)\big)\big(\pi_t^\mu(h)-\pi_t^\nu(h)\big)\ud t\\
	&\quad +  \big[\cdots\big]\ud I_t^\mu
\end{align*}
The claim follows from that $\pi_t^\mu(h) = \pi_t^\nu(\gamma_t h)$ and $\pi_t^\nu(\gamma_t) = 1$.

\subsection{Proof of Proposition~\ref{prop:backward-map-and-bsde}}\label{apdx:pf-prop-bmap-bsde}

Apply It\^o formula on $Y_t(X_t)$ to obtain
\begin{align*}
	\ud Y_t(X_t) &= V_t^\tp(X_t)\big(\ud Z_t-h(X_t)\ud t\big) + (\text{martingale associated with $\clA$})\\
	&=: V_t^\tp(X_t)\ud W_t + \ud N_t
\end{align*}
where $\{N_t:t\ge 0\}$ is a martingale associated with $\clA$. Integrating this from some $t$ to $T$ yields
\begin{equation}\label{eq:ytxt}
	\gamma_T(X_T) = Y_t(X_t) + \int_t^T V_s^\tp(X_s)\ud W_s +\ud N_t
\end{equation}
and therefore
\[
Y_t(x) = \E^\nu\big(\gamma_T(X_T)\mid \clZ_t \vee [X_t = x]\big),\quad x\in\bS
\]
In particular at time $t=0$, we have $Y_0(x) = y_0(x)$ by definition.

The variance of $Y_t(X_t)$ is also obtained from~\eqref{eq:ytxt}:
\begin{align*}
\E^\nu\big(|\gamma_T(X_T)-1|^2\big) &= \E^\nu\Big(|Y_t(X_t)-1|^2 + \int_t^T |V_s(X_s)|^2 + (\Gamma Y_s)(X_s)\ud s\Big)\\
&=\var^\nu(Y_t(X_t)) + \E^\nu\Big(\int_t^T \pi_s^\nu(\Gamma Y_s) + \pi_s^\nu(|V_s|^2)\ud s\Big)
\end{align*}
Hence we obtain~\eqref{eq:var_contractive_ext} by differentiation.

\subsection{Proof of Proposition~\ref{prop:functional-inequality}}\label{apdx:pf-filter-stability}

From~\eqref{eq:ytxt}, we note that $Y_t(X_t)$ is a martingale, and therefore $\E^\nu(Y_t(X_t)) = \E^\nu(\gamma_T(X_T)) = 1$ for all $0\le t \le T$.

Now set $\nu = \bmu$ from Prop.~\ref{prop:backward-map-and-bsde}, and then by c-PI we have
\[
\frac{\ud}{\ud t} \var^\bmu\big(Y_t(X_t)\big) \ge c\,\var^\bmu\big(Y_t(X_t)\big)
\]
Therefore,
\[
\var^\bmu\big(\gamma_T(X_T)\big) \ge e^{cT}\var^\bmu\big(Y_0(X_0)\big)
\]
Meanwhile from~\eqref{eq:chisq_inequality} and that $Y_0(x) = y_0(x)$,
\begin{align*}
	\E^\mu\big(\chisq(\pi_T^\mu\mid\pi_T^\bmu)\big)^2 &\le \var^\nu(Y_0(X_0))\var^\nu(\gamma_0(X_0))\\
	&\le e^{-cT}\var^\bmu\big(\gamma_T(X_T)\big) \var^\nu(\gamma_0(X_0))
\end{align*}
Since $\var^\bmu\big(\gamma_T(X_T)\big) = \E^\bmu\big(\chisq(\pi_T^\mu\mid\pi_T^\bmu)\big)$, we have
\[
R_T := \frac{\E^\mu\big(\chisq(\pi_T^\mu\mid\pi_T^\bmu)\big)}{\E^\bmu\big(\chisq(\pi_T^\mu\mid\pi_T^\bmu)\big)} \ge \underline{a}
\]
and
\[
R_T\E^\nu\big(\chisq(\pi_T^\mu\mid\pi_T^\bmu)\big)\le e^{-cT}\chisq(\mu\mid\bmu)
\]
Hence the claim is proved.

\subsection{Proof of the Theorem~\ref{thm:2nd-law-obs}}\label{apdx:pf-2nd-law-obs}
Upon  the application of It\^o's rule 
\begin{align*}
	\ud \mathcal F(\pi^\mu_t,U_t) 
	&=  \ud (\pi^\mu_t(U_t)) - \ud \mathcal S(\pi^\mu_t) \\&=    \pi^\mu_t(\partial_t U_t) \ud t + \ud \pi^\mu_t( U_t) \ud t+\ud \pi^\mu_t(\log \pi^\mu_t)  +\frac{1}{2}  \pi^\mu_t (h - \pi^\mu_t(h))^2\ud t
	\\
	&=   \pi^\mu_t(\partial_t U_t) \ud t + \pi^\mu_t(\mathcal A_t U_t) \ud t  + \pi^\mu_t(\mathcal A_t \log \pi^\mu_t) \ud t  +\frac{1}{2}  \pi^\mu_t (h - \pi^\mu_t(h))^2\ud t\\
	&\quad + [\cdots]\ud I_t^\mu
\end{align*}
Then, using the identity \[\pi^\mu_t(\mathcal A_t U_t +  \mathcal A_t \log (\pi^\mu_t) ) =  \nu_t( \gamma_t\mathcal A_t \log (\gamma_t)) = - \fish(\pi_t^\mu,
\nu_t),\] we arrive at 
\begin{align*}
	\ud \mathcal F(\pi^\mu_t,U_t) 
	&=  \pi^\mu_t(\partial_t U_t) \ud t - \fish(\pi_t^\mu,
	\nu_t) \ud t   +\frac{1}{2}  \pi^\mu_t (h - \pi^\mu_t(h))^2\ud t  + [\cdots] \ud I_t^\mu
\end{align*}
The result follows by taking the integral from $t=0$ to $t=T$ and, then, taking the expectation. 

%---------------------------------------------------------------
%
% BibTeX users please use
% \bibliographystyle{}
% \bibliography{}
%
% Non-BibTeX users please follow the syntax
% the syntax of "referenc.tex" for your own citations
%\input{referenc}
%%%%%%%%%%%%%%%%%%%%%%%%%%%%%%%%%%%%%%%%%%%%%%%%%%%%%%%%%%%%%%%%%%%%%%

\bibliographystyle{IEEEtran}
\bibliography{bibfiles/_master_bib_jin.bib,bibfiles/jin_papers.bib,bibfiles/thermo-refs.bib}

%%%%%%%%%%%%%%%%%%%%%%%%%%%%%%%%%%%%%%%%%%%%%%%%%%%%%%%%%%%%%%%%%%%%%%

\end{document}